\documentclass[12pt]{amsart}

\usepackage{amssymb}
\usepackage{amsmath}
\def\forkindep{\mathrel{\raise0.2ex\hbox{\ooalign{\hidewidth$\vert$\hidewidth\cr\raise-0.9ex\hbox{$\smile$}}}}}
\usepackage{ amsthm, amscd, amsfonts, amssymb, graphicx,tikz, color, environ}
\usepackage{xcolor}

\input xy

\xyoption{all}
\usepackage{setspace}
\usepackage[utf8]{inputenc}
\usepackage{amsrefs}

\def\forkindep{\mathrel{\raise0.2ex\hbox{\ooalign{\hidewidth$\vert$\hidewidth\cr\raise-0.9ex\hbox{$\smile$}}}}}
\usepackage{datetime}

\usepackage{amsmath}
\usepackage{amssymb}
\usepackage{amsthm}

\usepackage{fancyhdr}
\usepackage{amsfonts}
\usepackage{graphicx}
\usepackage{verbatim}
\usepackage[foot]{amsaddr}

\usepackage{enumerate}

\theoremstyle{plain}

\newtheorem{theorem}{Theorem}[section]

\newtheorem{proposition}[theorem]{Proposition}

\theoremstyle{definition} 

\newtheorem{definition}[theorem]{Definition}

\newtheorem{example}[theorem]{Example}

\newtheorem{remark}[theorem]{Remark}

\theoremstyle{remark}

\renewcommand{\phi}{\varphi}

\newcommand{\initial}\lessdot

\def\?{?\vadjust

{\vbox to 0pt{\vskip-7pt\hbox to 1.1\hsize{\hfill\huge ?!}}}}

\usepackage{verbatim}

 \def\nfork{\setbox0\hbox{$\bigcup$}%
 \setbox1=\hbox to \wd0{\hfil\vrule width 0.7pt depth 2pt height 7.5pt\hfil}%
 \wd1=0cm\relax\box1\box0}

\renewcommand{\epsilon}{\varepsilon}

\usepackage{enumerate}

\newcommand{\be}{\begin{enumerate}}
\newcommand{\ee}{\end{enumerate}}

\newcommand{\bd}{\begin{definition}}
\newcommand{\ed}{\end{definition}}

\begin{document}

\title{A Study of Abstract Elementary Classes in the context of Graphs}

\author{Navaneetha Madaparambu Rajan}
\email{navaneet.madapara@msmail.ariel.ac.il}
\address{Department of Mathematics\\ Ariel University \\ Ariel, Israel}

\maketitle

\begin{abstract}
 In the framework of graphs, we study abstract elementary classes (AECs). In this work we analyze several properties of $Forb(G)$ and versions of $Forb-Con(G)$ in the context of aecs and we present some examples of classes of graphs which contradicts amalgamation property.

\end{abstract} 


\tableofcontents

\section{Introduction}
Modern model theory began with Morley’s categoricity theorem\cite{article}: A first order theory is categorical in one uncountable cardinal $\kappa$ (has a unique model of that cardinality) if and only if it is categorical in all uncountable cardinals. This result triggered the change in emphasis from the study of logic to the study of theories. Keisler generalized Morley's theorem to this framework in addition that the model in the categoricity cardinal is sequentially homogeneous\cite{Keisler1971ModelTF}. Shelah later claimed that Keisler's assumption does not follow from categoricity by using an example in \cite{Marcus1972AMP}. Lessmann\cite{lessmann1997ranks} established that the categoricity in some uncountable cardinal implies categoricity in all uncountable cardinals, provided first-order theory $T$ is countable. When $T$ is uncountable, eventual categoricity conjecture for homogeneous model theory is implicit in Shelah's viewpoint and Hyttinen and Kesala\cite{hyttinen_kesälä_2011} proved more precisely that categoricity in some $\lambda > |T|$ with $\lambda \neq \aleph_{\omega}(|T|)$ implies categoricity in all $\lambda' \geq min(\lambda, h(|T|))$.
\par
Shelah splits all first order theories into $5$ classes. Many interesting algebraic structures fall into the three classes ($\omega$-stable, superstable, strictly stable) whose models admit a deep structural analysis\cite{shelah2009classification}. The model theory of the class of models of a sentence is one of a number of ‘non-elementary’ logic\cite{shelah2009classification}.
\par
The book \cite{shelah2009classification}, on elementary classes, i.e., classes of first order theories,
presents properties of theories, which are so called ‘dividing lines’ and investigates them. When such a property is satisfied, the theory is low, i.e.,
Structure theorems are verifiable. But when such a
property is not satisfied, we have non-structure, namely, there is a witness that the theory is complicated, and there are no structure theorems. This
witness can be the existence of many models in the same power.
\par
If we think of the study of first order theory $T$ as the study of the model class $\{M: M \models T\}$, then, at that point, the notion of abstract elementary classes is a generalization of that of first order theories. There
are notable results on first order theories, that are off-base or exceptionally hard to demonstrate with regards to AEC's. The failure of the Compactness Theorem is the primary cause.
\par
An abstract elementary class (AEC) $K$ is a collection of models and a notion of ‘strong submodel’ $\prec$ which satisfies general conditions similar to those satisfied by the class of models of a first order theory with $\prec$ as elementary submodel. In particular, the class is closed under unions of $\prec$-chains. A Lowenheim-Skolem number is assigned to each AEC: a cardinal $\kappa$ such that each $M\in K$ has a strong submodel of cardinality $\kappa$.
\par
Shelah introduced AECs\cite{inbook} to provide a uniform framework in which to generalize first-order classification theory. Shelah's categoricity conjecture\cite{shelah2009classification} states that : For every AEC $K$ there should be a cardinal $\mu$ depending only on $LS(K)$ such that if $K$ is categorical in some $\lambda \geq \mu$ (i.e. K has exactly one (up to isomorphism) model of size $\lambda$), then $K$ is categorical in $\theta$ for all $\theta \geq \mu$. In other words, for every AEC $K$, there exists a cardinal $\lambda_0$ and there exists $l<2$ such that for every $\lambda \geq \lambda_0$, $I(\lambda,K) = 1$ if and only if $l = 0$. Similarly, we have Shelah's main gap conjecture\cite{inbook} which states : for every AEC $K$, there exists a cardinal $\lambda_0$ and there exists $l<2$ such that for every $\lambda \geq \lambda_0$, $I(\lambda,K) = 2^\lambda$ if and only if $l = 1$. Los Conjecture 1954\cite{o1954OnTC} and Morley 1965\cite{article} states that : Let $T$ be a first order theory. If there exists $\lambda > |T|+\aleph_0$ such that $I(\lambda, T)=1$, then $I(\mu, T)=1$ holds for every $\mu > |T|+\aleph_0$.
\par
We generally consider the class of all structures in a fixed vocabulary where the structure is a set with defined interpretations for the relation, function and constant symbols of the vocabulary.
\par
In the section 2, we provide sufficient background on abstract elementary classes, results related to them, and an overview of AECs of graphs. Various examples of classes of graphs that contradicts the amalgamation property are given in section 3. In section 4 and 5, we examine the properties of $Forb(G)$ and variants of $Forb-con(G)$. 

\section{Preliminaries}
We notate integers or natural numbers by alphabets $n,m$ ordinal numbers by $\alpha, \beta,\gamma, i, j$ limit ordinal number by $\delta$ and cardinal numbers by $\kappa, \lambda, \mu$.

\begin{definition}(Abstract Elementary Class)\label{def of Aec}\cite{jarden2013non}
Let $K$ be any class of models for a fixed vocabulary and let $\preceq$ be a 2-place relation on $K$. The pair $(K,\preceq)$ is an AEC if the following axioms are satisfied:
\begin{enumerate}
\item Closed under isomorphisms: If $M_1\in K$, $M_0\preceq M_1$ and $f:M_1\longrightarrow N_1$ is an isomorphism, then $N_1\in K$ and $f[M_0]\preceq f[M_1]=N_1$.
\item $\preceq$ is a partial order on $K$ and it is included in the inclusion relation.
\item Tarski-Vaught Axioms: 
    \begin{enumerate}

        \item If $\langle M_\alpha : \alpha < \delta \rangle$ is a $\preceq$-increasing continuous sequence, then $M_0 \preceq \bigcup \{ M_\alpha : \alpha < \delta\} \in K$.
        \item (Smoothness)If $\langle M_\alpha : \alpha < \delta \rangle$ is a $\preceq$-increasing continuous sequence, and for every $\alpha<\delta$, $M_\alpha \preceq N$, then $\bigcup \{M_\alpha : \alpha < \delta \} \preceq N$.\end{enumerate}
\item Coherence: If $M_0\subseteq M_1\subseteq M_2$ and $M_0\preceq M_2 \wedge M_1\preceq M_2$, then $M_0\preceq M_1$.
 \item LST condition: There exists a Lowenheim Skolem Tarski number, $LST(K,\preceq)$, which is the first cardinal $\lambda$, such that for every model $N\in K$ and a subset $A$ of it, there is a model $M\in K$ such that $A\subseteq M \preceq N$ and $||M|| \leq \lambda + |A|$.
\end{enumerate}
\end{definition}

\begin{definition}(Amalgamation Property-AP)\cite{jarden2013non}
    Let $n<3$, for every $M_0, M_1, M_2 \in K_\lambda$ with $M_0\preceq M_n$, there are $f_1, f_2, M_3$ such that $f_n : M_n\longrightarrow M_3$ is an embedding over $M_0$, i.e., the diagram below commutes. Then, we say that $(K_\lambda,\preceq \restriction{K_\lambda})$ satisfies amalgamation property.
    \[
\xymatrix{
M_1\ar[r]^{f_1} &M_3\\
M_0\ar[u]^{id} \ar[r]^{id} &M_2\ar[u]^{f_2}
}
\]
    
We say that $(f_1, f_2, M_3)$ is an amalgamation of $M_1$ and $M_2$ over $M_0$ or $M_3$ is an amalgam of $M_1$ and $M_2$ over $M_0$.
\par
And,$(f_1,f_2, M_3)$ is a disjoint amalgamation of $M_1$ and $M_2$ over $M_0$ when the images of the embedding intersects and their intersection is $M_0$. 
\end{definition}

\begin{definition}(Joint Embedding Property-JEP)\cite{jarden2013non}
If $M_1, M_2 \in K_\lambda$, then there are $f_1, f_2$ and $M_3$ such that for $n = 1, 2$ $f_n : M_n \longrightarrow M_3$ is an embedding and $M_3 \in K_\lambda$. Then, we say that $K_\lambda$ satisfies joint embedding property.
\end{definition}

\bd
A \emph{graph} is a model $G=(V,E)$ of the following theory:
\begin{align*}
&\forall x,y [E(x,y) \implies x \neq y] \\
&\forall x,y [E(x,y) \implies E(y,x)]
\end{align*}
We define a binary relation $\preceq$ on graphs as follows:
$G_1 \preceq G_2$ holds when: $G_1 \subseteq G_2$ ($G_1$ is a substructure of $G_2$) and for each $x \in G_2 \setminus G_1$ there is an element $y \in G_1$ such that $x,y$ are in the same connected component of $G_2$ and if it holds for two elements $y,z \in G_1$ then they are of the same connected component in $G_1$.
\ed

\begin{proposition}
Let $K^G$ be the class of graphs. The pair $(K^G,\preceq)$ forms an AEC.
\end{proposition}

\begin{proof}

Clearly $(K^G,\preceq)$ satisfies all axioms of an AEC with $LST(K^G)=\aleph_0$. 
\end{proof}

\begin{proposition}
The AEC $(K^G,\preceq)$ satisfies the amalgamation property, joint embedding property and no maximal models.\end{proposition} 
\begin{proof}
Take the disjoint union of the graphs to be the amalgam.
\end{proof}

\begin{remark}
The class of complete graphs along with the relation $\preceq$ restricted to it forms a sub-AEC of $(K^G,\preceq)$.
\end{remark}

\begin{proposition}

The class of graphs $K^G$ is a universal class.\end{proposition}
\begin{proof}
Since $(K^G,\preceq)$ is an aec, $K^G$ is closed under isomorphisms and  closed under $\subseteq$-increasing chains. If $G \in K^G$ and $H \subseteq G$, then clearly $H \in K$.
\end{proof}

\section{Classes of Graphs Contradicting the Amalgamation Property}
It is easy to produce a class of graphs that contradicts the amalgamation property. For example, fix two types, $p,q$ over a graph $G_0$. Let $K$ be the class of graph that do not realize $p,q$ together. They may satisfy one of them but not both. The AEC $(K,\subseteq)$ does not satisfy the amalgamation property.
\par
We give few more examples of class of graphs which does not satify amalgamation property. 

\begin{example}Consider the class of graphs $K$ whose each connected component has no more than $n$ elements. Then, $K$ does not satisfy amalgamation property.
\end{example}

\begin{example}
Let $K$ be the class of graphs $G$ such that if $G$ has at least $k$ connected components then each connected component has at most $n$ elements. The amalgamation property does not hold in this case.
\end{example}

\begin{example}
Fix three graphs: $G_0^*,G_1^*,G_2^*$.  Let $K$ be the class of graphs $G$ such that at least one of the three fixed graphs is not embedded in $G$. Then, the amalgamation property fails.\\ 
We can replace 3 by another number, maybe an infinite number of graphs.\end{example}

The following is an example of class of graphs which is not an AEC.
\begin{example}
The class $K$ of graphs $G$ such that if $G$ has infinite number of connected components then each one has no more than n elements. Then $(K,\subseteq)$ is not an AEC.
\end{example}

\begin{remark} Fix a sequence of graphs: $\langle G^*_\alpha:\alpha<\alpha^* \rangle$.  Let $G_\alpha \implies G_\beta$ be the following sentence: `if $G_\alpha$ is embedded in $G$ then $G_\beta$ is embedded in $G$. In this way, we can define a set of sentences like $G_\alpha \iff G_\beta$, $G_\alpha \vee G_\beta$, $G_\alpha \wedge G_\beta$ etc. Each subset A of this set of `sentences' give rise to the class of graph satisfying them.
\end{remark}

\section{Forb(G)}

\begin{definition}
Let $\mathbf{G}$ be a family of graphs ('forbidden subgraphs'). We define $Forb(\mathbf{G})$ as the class of graphs containing no induced subgraph isomorphic to any element of $\mathbf{G}$.
\end{definition}

\begin{remark}
If we consider $\mathbf{G} \subset \{G:G$ is connected$\}$ then $Forb(\mathbf{G})$ has joint embedding property. But, amalgamation property fails to hold here.
\end{remark}

Assume that the graphs in $\mathbf{G}$ are of cardinality $\lambda$ at least. So the joint embedding and amalgamation property even over sets is satisfied in $K_{<\lambda}$. But there is a problem: the class is not closed under union of increasing sequences if the union has cardinality $\lambda$. If we want to study only models of cardinality less than $\lambda$, then there is no point in forbidding graphs of cardinality $\lambda$ at least. So we should think on a new relation $\preceq$.
\par
\begin{definition}

 For each $G$ in $\mathbf{G}$, we fix a cardinal $\lambda_G$ smaller than $|G|$. Now we define the relation $\preceq$ as follows: $M \preceq N$ where $M$ is an induced subgraph of $N$, they are both in $Forb(\mathbf{G})$ and for each induced subgraph $H$ of $G$ of cardinality $\lambda_G$ at least if $h:H \to M$ is an embedding then we can't strictly extend $h$ to an embedding of some $H'$ between $H$ and $G$ into $N$.
 \end{definition}
The idea is that if we have $\lambda_G$ elements of $G$ then it is forbidden to add elements of $G$. In this way, we avoid the case where in the union of an increasing sequence of graphs we will have a copy of $G$.
\par
\begin{proposition}
$Forb(\mathbf{G})$ with the relation $\preceq$ is closed under union of increasing sequence of graphs.
\end{proposition}
\begin{proof}

  Let $\langle M_\alpha:\alpha<\delta \rangle$ be an increasing sequence of graphs in $Forb(\mathbf{G})$ and let $M_\delta$ be its union. Assume for the sake of contradiction that $M_\delta \notin Forb(\mathbf{G})$, namely for some $G \in \mathbf{G}$ there is an embedding $h:G \to M_\delta$. For each $v \in G$ let $\alpha(v)$ be the first ordinal $\alpha$ such that $h(v) \in M_\alpha$.\\
Case 1: For each $\alpha<\delta$ the set $G_\alpha:=\{v \in G:\alpha(v) \leq \alpha\}$ has cardinality less than $\lambda_G$. In this case, the sequence $\langle G_\alpha:\alpha<\delta \rangle$ is an increasing sequence of sets of cardinality less than $\lambda_G$ whose union has size bigger than $\lambda_G$, which is impossible.
\\
Case 2: For some $\alpha$, the set $G_\alpha$ is big. In this case, $h$ must be its restriction to $M_\alpha$. So $G$ is embedded in $M_\alpha$ and $M_\alpha$ is not in $Forb(\mathbf{G})$, a contradiction.
\end{proof}
\par
\begin{proposition}
The pair $(Forb(\mathbf{G}),\preceq)$ does not satisfy smoothness.
\end{proposition}
\begin{proof}
Assume that $|G|>\aleph_0,\lambda_G \geq \aleph_0$. Fix $G^- \subseteq G$ with $|G^-|=\lambda_G$. Enumerate $G^-$ as $\{V_\alpha:\alpha<\lambda_G\}$. Let $M_\alpha$ be the induced subgraph of $G^-$ whose universe is $\{V_\beta:\beta<\alpha\}$. So, $\langle M_\alpha:\alpha<\delta \rangle \}$ is increasing(since $||M_\alpha||<\lambda_G$ so, every induced subgraph $G'$ of $G$ is embedded in $M_\alpha$). Let $N$ will be such that $G^- \subset N \subset G$.Then, $M_\delta=\bigcup\limits_{\alpha<\delta}M_\alpha=G^-\npreceq N$.
\end{proof}
\par 

So we have to return to $(Forb(\mathbf{G}),\subseteq)$ (namely to the relation of being an induced subgraph) where each $G$ in $F$ is finite.

We try to find $G$ that admits an AEC with JEP and LST-number less than $|G|$. It means that there is a relation $\preceq$ on $Forb(G)$ such that $(Forb(G),\preceq)$ is an aec with JEP and its LST-number is less than $|G|$. 

\bd
Fix two graphs $G$ and $M$. We define \emph{the number of elements of $G$ in $M$} as:
\[
\sup\{|h|:h \text{ is an embedding of an induced subgraph of } G \text{ into } M\}.
\]
\ed

\begin{example}
Let $G$ be the graph with $\lambda^+$ vertices and no edges. Let $\preceq$ be define as follows: $M \preceq N$ iff $M \subseteq N$ and the number of elements of $G$ in $N$ equals the number of elements of $G$ in $M$. In this case, the relation is not closed under increasing sequences. Indeed, Let $\langle M_\alpha:\alpha<\lambda^+ \rangle$ be a filtration of $G$ such that $||M_\alpha||=\lambda$ for each $\alpha<\lambda^+$. Each $M_\alpha$ has $\lambda$ vertices and no edges. Clearly, $M_\alpha \preceq M_{\alpha+1}$ for each $\alpha<\lambda^+$. But the union $\bigcup_{\alpha<\lambda^+}M_\alpha$ is $G$ which is not in $Forb(G)$. 
\end{example}

\begin{example}
Let $G$ be the graph with $\lambda^+$ vertices and no edges and let $\kappa$ be an infinite cardinal which is $\leq \lambda^+$. Let $\preceq$ be define as follows: $M \preceq N$ iff $M \subseteq N$ and $N$ does not add an element of $G$ over $\kappa$ elements of $M$. 

We show that $(Forb(G),\preceq)$ does not satisfy smoothness. Let $M$ be a graph with exactly  $\kappa$ many vertices and no edges. Let $\langle M_\alpha:\alpha < \kappa \rangle$ be a filtration of $M$. Let $N$ be the graph obtained from $M$ by adding an element $b$ but no edges. For each $\alpha<\kappa$ we have $M_\alpha \preceq N$ because $M_\alpha$ has less than $\kappa$ elements. But $M \npreceq N$.  
\end{example}

\begin{example}\label{example no edges not add an element}
Let $G$ be the graph with $\lambda^+$ vertices and no edges and let $1 \leq n<\omega$. Let $\preceq$ be define as follows: $M \preceq N$ iff $M \subseteq N$ and $N$ does not add an element of $G$ over $n$ elements of $M$. 

Then $(Forb(G),\preceq)$ has JEP(Let $M$ and $N$ be the graphs for which we want to find a joint embedding. It is enough to take the disjoint union of $M$ and $N$ and connect each vertex of $M$ with every vertices in $N$ by an edge). 
We prove that there is no $LST$-number.

Let $N$ be the following graph: its universe (set of vertices), $V_N$, is some infinite cardinal $\mu$. Its set of edges is:
\[
E_N=\{(\alpha,\beta):n\leq \alpha,\beta<\mu \text{ and } \alpha+n\leq \beta \vee \beta+n\leq \alpha\}
\]
Let $A=\{0,1,\dots n-1\}$. Let $M \in Forb(G)$ such that $A \subseteq M \preceq N$.

$N$ belongs to $Forb(G)$. Moreover, every induced subgraph $N'$ of $N$ of cardinality $2n+1$ at least has an edge, so cannot be embedded in $G$. Take a subset $B$ of $|N'|$ of cardinality $n+1$ at least whose all elements are not in $\{0,1,\dots n-1\}$. We can find two elements $\alpha,\beta \in B$ so that $\alpha+n<\beta$, so $(\alpha,\beta)$ is an edge in $N'$.

We now show that $M=N$. We argue by contradiction. Assume that there exists $x \in N \setminus M$. So the induced subgraph of $N$, $N[A \cup \{x\}]$ has no edges. So we find an embedding $h:N[A \cup \{x\}] \to G$. Hence, $N$ adds an element of $G$ over $M$, a contradiction.

\end{example}

The following example generalizes Example \ref{example no edges not add an element}.

\begin{example}\label{example not add an element}
Let $G$ be the graph with $\lambda^+$ vertices and let $1 \leq n<\omega$ be such that for each two finite induced subgraphs $G_1=(V_1,E_1)$ and $G_2=(V_2,E_2)$ of $G$ of cardinality $n$ at most, there is a graph $G_3=(V_3,E_3)$ such that the following hold:
\be
\item
 $V_3=V_1 \cup V_2$;
\item for each $v \in V_1$, the induced subgraph  $V_2 \cup \{v\}$ of $G_3$ is not embedded in $G$ and
\item for each $v \in V_2$, the induced subgraph $V_1 \cup \{v\}$ of $G_3$ is not embedded in $G$.
\ee 
Let $\preceq$ be define as follows: $M \preceq N$ iff $M \subseteq N$ and $N$ does not add an element of $G$ over $n$ elements of $M$. 

Then $(Forb(G),\preceq)$ has JEP(take the disjoint union). 

\end{example}

\begin{remark}\label{remark_for_next}
Let $G$ be a graph and let $n<\omega$. If for each induced subgraph $A$ of $G$ of cardinality $n$, every two elements $x,y \in G \setminus A$ we have $tp_{L_{\omega,\omega}}(x/A;G)=tp_{L_{\omega,\omega}}(y/A;G)$ then $G$ is complete or has no vertices. 
\end{remark}

\begin{theorem}
Let $n<\omega$ and $G$ be an infinite graph. Let $\preceq_n$ be the following partial order on $Forb(G)$: $M \preceq_n N$ when: for each induced subgraph $A$ of $M$ and each $x \in N$ if $tp_{L_{\omega,\omega}}(x/A;N)$ is realized in $G$ then $x \in M$.
Then $(Forb(G),\preceq_n)$ has no LST-number.
\end{theorem}
\begin{proof}
\emph{Case A:} $G$ is complete or has no edges. By the Remark \ref{remark_for_next}.\\
\emph{Case b:} $G$ is not complete, but has at least one edge. We can find an induced subgraph $A$ of $G$ of size $n$ and two distinct first order types $p,q$ over $A$ that are realized in $G$. It is easy to construct a graph $N$ such that $A$ is an induced subgraph of $N$ and the types of each $x \in N \setminus A$ is $p$. Since $q$ is not realized in $N$, there is no embedding of $G$ in $N$, so $N \in Forb(G)$.
\par
We prove that if $A \subseteq M$ and $M \preceq_n N$ then $M=N$. Let $x \in N$. Since $tp_{L_{\omega,\omega}}(x/A;N)=p$ is realized in $G$, by the definition of $\preceq_n$, we have $x \in M$.
\end{proof}
\par
These findings raise the possibility that there isn't an aec in this particular situation.

\section{Versions of Forb-con(G)}

In this section, we try to define two different relations on the class of graphs of $Forb-con(G)$ and check whether they form an aec.
\bd
 Let $G$ be a connected graph. Define $Forb-con(G)$ as the class of graphs $M$ such that each connected component of $M$ is not isomorphic to $G$.
We define $M \preceq N$ as $M \subseteq N$ and for each  complete graph $C$ that is embedded in $G$ if $C \cap M \neq \emptyset$ then $C \cap N \subseteq M$.
\ed

\begin{proposition}\label{proposition forb-con is an aec}
The pair $(Forb-con(G),\preceq)$ satisfies JEP.
\end{proposition}

\begin{proof}
Take disjoint union.
\end{proof}

\begin{proposition}
$(Forb-con(G),\preceq)$ satisfies transitivity.
\end{proposition}

\begin{proof}
  Assume that $M_0 \preceq M_1 \preceq M_2$. We have to show that $M_0 \preceq M_2$. Let $C$ be a complete graph that is embedded in $G$. Assume that $C \cap M_0 \neq \emptyset$. So $C \cap M_1 \neq \emptyset$. Since $M_0 \preceq M_1 \preceq M_2$, we have $C \cap M_1 \subseteq M_0$ and $C \cap M_2 \subseteq M_1$. Hence, $C \cap M_2 \subseteq M_0$.
\end{proof}

\begin{proposition}
The pair $(Forb-con(G),\preceq)$ satisfies the following strong version of coherence: If $M_0 \subseteq M_1 \subseteq M_2$ and $M_0 \preceq M_2$ then $M_0 \preceq M_1$.
\end{proposition}

\begin{proof}
Let $C$ be a complete graph that is embedded in $G$ such that $C \cap M_0 \neq \emptyset$. Then $C \cap M_2 \subseteq M_0$ and so $C \cap M_1 \subseteq M_0$.
\end{proof}

\begin{proposition}
$(Forb-con(G),\preceq)$ satisfies Axiom (3)(a) of aecs.
\end{proposition}

\begin{proof}
We should prove that $M_0 \preceq M_\delta$. Let $C$ be a complete graph that is embedded in $G$ and assume that $C \cap M_0 \neq \emptyset$. So $C \cap M_\alpha \subseteq M_0$ for each $\alpha<\delta$. Therefore $C \cap M_\delta \subseteq M_0$.
\end{proof}

\begin{proposition}
$(Forb-con(G),\preceq)$ satisfies smoothness.
\end{proposition}

\begin{proof}
We have to prove that $M_\delta \preceq N$. Let $C$ be a complete graph that is embedded in $G$ such that $C \cap M_\delta \neq \emptyset$. So for some $\alpha<\delta$, we have $C \cap M_\alpha \neq \emptyset$. Since $M_\alpha \preceq N$, we have $C \cap N \subseteq M_\alpha \subseteq M_\delta$.
\end{proof}

Now we define another version of relation for $Forb-con(G)$.

\bd
We define $Forb-con(G)$ as the class of graphs $M$ such that each connected component of $M$ is not isomorphic to $G$. We define $M \preceq N$ iff $M \subseteq N$ and every connected component $C$ of $M$ which is embedded in $G$ is a connected component of $N$. 
\ed

\begin{proposition}
$(Forb-con(G),\preceq)$ satisfies transitivity.
\end{proposition}

\begin{proof}
  Assume that $M_0 \preceq M_1 \preceq M_2$. We have to show that $M_0 \preceq M_2$. Let $C$ be a connected component of $M_0$ that is embedded in $G$. Then $C$ is a connected component of $M_1$ and so a connected component of $M_2$. 
\end{proof}

\begin{proposition}
The pair $(Forb-con(G),\preceq)$ satisfies the following strong version of coherence: If $M_0 \subseteq M_1 \subseteq M_2$ and $M_0 \preceq M_2$ then $M_0 \preceq M_1$.
\end{proposition}

\begin{proof}
Let $C$ be a connected component of $M_0$ that is embedded in $G$. Let $x \in M_1 \setminus M_0$. If $x$ is adjoint to some element in $C$ then $C$ is not a connected component of $M_2$, in a contradiction the assumption $M_0 \preceq M_2$. So $C$ is a connected component of $M_1$.
\end{proof}

\begin{proposition}
$(Forb-con(G),\preceq)$ satisfies Axiom (3)(a) of aecs.
\end{proposition}
\begin{proof}
We should prove that $M_0 \preceq M_\delta$. Let $C$ be a connected component of $M_0$ that is embedded in $G$. Toward a contradiction assume that some $x \in M_\delta \setminus M_0$ is adjoined to some element in $C$. Fix $\alpha<\delta$ so that $x \in M_\alpha$. Since $M_0 \preceq M_1$, $C$ is a connected component of $M_\alpha$, a contradiction.
\end{proof}

\begin{proposition}
$(Forb-con(G),\preceq)$ satisfies smoothness.
\end{proposition}

\begin{proof}
We have to prove that $M_\delta \preceq N$. Let $C$ be a connected component of $M_\delta$ that is embedded in $G$. We shall prove that $C$ is a connected component of $N$. Let $\alpha<\delta$ be an element such that $C \cap M_\alpha \neq \emptyset$. Since $C \cap M_\alpha$ is a connected component of $M_\alpha$ that is embedded in $G$ and $M_\alpha \preceq N$, the set $C \cap M_\alpha$ is a connected component of $N$. Hence, $C$ is included in $M_\alpha$ and $C=C \cap M_\alpha$ is a connected component of $N$. 
\end{proof}
\par
In both $Forb-con(G)$ scenarios, we are unable to handle the LST number. A few challenges surfaced, which might be covered in later works.

\section{Acknowledgement}
I want to sincerely thank Dr. Adi Jarden for all of his assistance during this research initiative, including his mentorship, support, and guidance. His knowledge, perceptions, and helpful criticism have greatly influenced the course and caliber of this study.

\bibliography{Bibliography.bib}

\end{document}